\documentclass{amsart}
\usepackage[utf8]{inputenc}

\usepackage{todonotes}
\usepackage[asymmetric,vmargin=3cm,hmargin=3cm,marginpar=2cm]{geometry}
\usepackage{amsmath,amssymb,amsfonts,amsthm,amscd,graphicx,psfrag,epsfig,fge}
\DeclareMathOperator{\Li}{Li}

\usepackage[]{hyperref}
\usepackage{mathtools}
\usepackage{thmtools}

\newtheorem{theorem}{Theorem}
\newtheorem{proposition}[theorem]{Proposition}
\newtheorem{conjecture}[theorem]{Conjecture}
\newtheorem{corollary}[theorem]{Corollary}

\title[On the Goncharov depth conjecture and polylogarithms of depth two]{On the Goncharov depth conjecture \\ and polylogarithms of depth two}
\author[Charlton]{Steven Charlton}
\address{Fachbereich Mathematik (AZ), Universit\"at Hamburg, Bundesstra\textup{\ss}e 55, 20146 Hamburg, Germany}
\email{steven.charlton@uni-hamburg.de}

\author[Gangl]{Herbert Gangl}
\address{Department of Mathematical Sciences, Durham University, Durham DH1 3LE, United Kingdom}
\email{herbert.gangl@durham.ac.uk}

\author[Radchenko]{Danylo Radchenko}
\address{Laboratoire Paul Painlev\'e, Universit\'e de Lille, F-59655 Villeneuve d'Ascq,	France}
\email{danradchenko@gmail.com}

\author[Rudenko]{Daniil Rudenko}
\address{Department of Mathematics, University of Chicago, 5801 S Ellis Ave, 60637 Chicago, IL, USA}
\email{rudenkodaniil@gmail.com}
\date{\today}

\makeatletter
\renewcommand{\sectionautorefname}{\S{}\@gobble}
\makeatother
\def\equationautorefname~#1\null{(#1)\null}

\begin{document}

\maketitle
\begin{abstract}
 We prove the surjectivity part of Goncharov's depth conjecture. We also show that the depth conjecture implies that multiple polylogarithms of depth $d$ and weight $n$ can be expressed via a single function $\Li_{n-d+1,1,\dots,1}(a_1,a_2,\dots,a_d)$, and we prove this latter statement for $d=2$.  
\end{abstract}

\section{Introduction}\label{SectionIntroduction}

Multiple polylogarithms are multivalued functions of variables
$a_1,\dots,a_d\in \mathbb{C}$ 
depending on positive integer parameters $n_1,\dots,n_d\in \mathbb{N}.$ In the polydisc   $|a_1|,|a_2|,\dots, |a_d| <1$ polylogarithms are defined by power series 
\[
\Li_{n_1,n_2,\dots, n_d}(a_1,a_2,\dots,a_d)=\sum_{0<m_1<m_2<\dots<m_d}\frac{a_1^{m_1} a_2^{m_2}\dots a_d^{m_d}}{m_1^{n_1}m_2^{n_2}\dots m_d^{n_d}}\,.
\]
The number $n=n_1+\dots+n_d$ is called the weight of the multiple polylogarithm, and the number $d$ is called its depth. Goncharov suggested an ambitious conjecture, giving a necessary and sufficient condition for a sum of polylogarithms to have certain depth. In \autoref{SectionSurjectivity} we show that the Goncharov depth conjecture would have the following remarkable corollary.

\begin{conjecture}\label{ConjectureMain}
Any multiple polylogarithm 
of weight $n\geq 2$ and depth $d\geq 2$ can be expressed as a linear combination of multiple polylogarithms 
$\Li_{n-d+1,1,\dots,1}$
and products of polylogarithms of lower weight.
\end{conjecture}
We expect that there exists a presentation where all the arguments are Laurent monomials in $\sqrt[N]{a_1}, \dots, \sqrt[N]{a_d}$ for sufficiently large $N.$  We show that \autoref{ConjectureMain} is true for $d=2$.

\begin{theorem}\label{TheoremOne} For every $0<k<n\in\mathbb{N}$ there exists $N\in\mathbb{N}$ such that the multiple polylogarithm $\Li_{k,n-k}(x,y)$  can be expressed as a linear combination of functions 
\[
\Li_{n-1,1}(\sqrt[N]{\mathllap{\phantom{xy}}x}^{\,r} \sqrt[N]{\mathllap{\phantom{xy}}y}^{\,s},\sqrt[N]{\mathllap{\phantom{xy}}x}^{\,t} \sqrt[N]{\mathllap{\phantom{xy}}y}^{\,u} ) \quad \text{\ for\ \  } r,s,t,u \in \mathbb{Z}
\]
and products of classical polylogarithms, where each appearance of $\root{N}\of{z}$ denotes any $N$th root of $z$.
\end{theorem}

Here is an example of this type of identity in weight four and depth two
   \begin{align*}
      \Li_{2,2}(x,y)= {} & 
       {-}4 \Li_{3,1}\Big({-}\frac{\sqrt{x}}{\sqrt{y}},y\Big)-4 \Li_{3,1}\Big(\frac{\sqrt{x}}{\sqrt{y}},y\Big)
      +4 \Li_{3,1}\Big({-}\frac{\sqrt{y}}{\sqrt{x}},x\Big)+4 \Li_{3,1}\Big(\frac{\sqrt{y}}{\sqrt{x}},x\Big)
      \\
      &
     {} + \Li_{3,1}(x,y)-\Li_{3,1}(y,x)
       -\Li_{3,1}\Big(\frac{y}{x},x\Big)-\frac{1}{2} \Li_4(x y)+\Li_1(x) \Li_3(y) \,.
   \end{align*}

In \autoref{SectionProofTheoremOne} we give an elementary proof of \autoref{TheoremOne}. In \autoref{SectionSurjectivity} we recall the statement of Goncharov's depth conjecture and prove a part of it (\autoref{TheoremTwo}). Next, we show that the depth conjecture implies \autoref{ConjectureMain}.

\section{Proof of \autoref{TheoremOne}}\label{SectionProofTheoremOne}
We define $L(x,y\mid t_1,t_2)$ to be the following generating function
    \[L(x,y\mid t_1,t_2) := \sum_{k,l>0}\Li_{k,l}(x/y,y)t_1^{k-1}t_2^{l-1} = \sum_{m,n>0}\frac{x^my^n}{(m-t_1)(m+n-t_2)}\,.\]
The key observation used in the proof of \autoref{TheoremOne} is the following identity.
\begin{proposition}
	For any integers $\alpha,\beta,\gamma>0$ with $\gamma=\alpha+\beta$ and any $x,y$ with $|x|,|y|<1$ we have
	\begin{multline} \label{eq:main1}
		\sum_{\substack{X^{\alpha}=x, Y^{\beta}=y,\\ Z^{\gamma}=xy}}\bigg[
		   \frac{1}{\gamma} L(X,Y\mid    \alpha\beta t, 0)
		  -\frac{1}{\alpha} L(Z,Y\mid    \gamma\beta t, 0)
		  +\frac{1}{\beta}  L(Z,X\mid {-}\gamma\alpha t, 0)\bigg]\\[-3ex]
		= L(xy,x\mid {-}\alpha t,\beta t) + \frac{1}{\gamma t}\sum_{k\ge 2}\Li_{k}(xy)(\beta t)^{k-1}\,.
	\end{multline}
\end{proposition}
\begin{proof}
Note that
	\[\sum_{\substack{X^\alpha=x\\ Y^{\beta}=y}}L(X,Y\mid t_1,t_2)=
	\sum_{m,n>0}\frac{\alpha\beta x^{m}y^{n}}{(m\alpha-t_1)(m\alpha+n\beta-t_2)}\,.\]
From this we calculate that the LHS of~\eqref{eq:main1} is equal to
\begin{align*}
	& \sum_{m,n>0}
	\Big[\frac{\beta x^{m}y^{n}}{(m-\beta t)(m\alpha+n\beta)}
	-\frac{\beta x^my^{m+n}}{(m-\beta t)(m\alpha+(m+n)\beta)}
	+\frac{\alpha x^{m+n}y^{m}}{(m+\alpha t)((m+n)\alpha+m\beta)}\Big]\\
	&  \hspace{4em}  {} = \sum_{m=n>0}\frac{\beta x^{m}y^{n}}{(m-\beta t)(m\alpha+n\beta)}
	+ \sum_{m>n>0}\Big[\frac{\beta x^{m}y^{n}}{(m-\beta t)(m\alpha+n\beta)}
	+\frac{\alpha x^{m}y^{n}}{(n+\alpha t)(m\alpha+n\beta)}\Big]\\
	& \hspace{4em} {}= \frac{1}{\gamma t}\sum_{n>0}\Big[\frac{(xy)^{n}}{n-\beta t}-\frac{(xy)^{n}}{n}\Big]
	+\sum_{m>n>0}\frac{x^{m}y^{n}}{(m-\beta t)(n+\alpha t)}\\
	&  \hspace{4em}  {}= \frac{1}{\gamma t}\sum_{k\ge 2}\Li_k(xy)(\beta t)^{k-1} + L(xy,x\mid -\alpha t,\beta t)\,. 
	 \qedhere
\end{align*}
\end{proof}

\begin{proof}[Proof of \autoref{TheoremOne}]
	Expanding both sides of~\eqref{eq:main1} as a power series in~$t$ and comparing the coefficients of $t^{n-2}$ 
	we see that for any integers $\alpha,\beta>0$ the function
	\[U_{n}^{\alpha,\beta}(x,y) := \sum_{\substack{k+l=n,\\k,l>0}}\Li_{k,l}(y,x)(-\alpha)^{k-1}\beta^{l-1}\]
	is expressible in terms of $\Li_{n-1,1}$ and $\Li_{n}$. Since the matrix $((-i)^{k-1}(n-i)^{n-d-1})_{i,k=1}^{n-1}$
	is invertible (its determinant is of Vandermonde type), each individual function $\Li_{k,l}(y,x)$ for $k+l=n$ can be written 
	as a rational linear combination of the functions $U_n^{1,n-1}(x,y),U_n^{2,n-2}(x,y),\dots,U_n^{n-1,1}(x,y)$, and hence it also 
	can be expressed in terms of $\Li_{n-1,1}$ and $\Li_{n}$, as claimed.
\end{proof}

\section{Surjectivity part of the Goncharov depth conjecture}\label{SectionSurjectivity}

To state the Goncharov depth conjecture we recall the definition of the Lie coalgebra $\mathcal{L}_\bullet(\mathrm{F})$ of (formal) polylogarithms with values in a field $\mathrm{F}$ (\cite{GoncharovMSRI}, see also \cite[\S2.1]{MR22}).  The Lie coalgebra  $\mathcal{L}_\bullet(\mathrm{F})$ is positively graded by weight; the component of weight $n$  is generated over $\mathbb{Q}$ by formal symbols $\Li^{\mathcal{L}}_{n_0\,;\,n_1,\dots, n_d}(a_1,\dots,a_d)$ for $n_0\in \mathbb{Z}_{\geq0}, n_1,\dots, n_d\in \mathbb{N}$ with $n_0+n_1+\dots+n_d=n$ and $a_1,\dots,a_d\in \mathrm{F}^{\times},$ which are subject to (mostly unknown) functional equations for polylogarithms. The cobracket $\Delta \colon\mathcal{L}_\bullet(\mathrm{F})\longrightarrow \bigwedge^2 \mathcal{L}_\bullet(\mathrm{F})$ was discovered by Goncharov (\cite{GoncharovMSRI}, \cite{Gon95B}, \cite{Gon01}); the definition was inspired by properties of mixed Hodge structures related to multiple polylogarithms. The Lie coalgebra $\mathcal{L}_\bullet(\mathrm{F})$ is filtered by depth; denote by $\mathcal{D}_d \mathcal{L}_\bullet(\mathrm{F}) $ the subspace spanned by polylogarithms of depth not greater than $d;$ let $\textup{gr}_d^{\mathcal{D}} \mathcal{L}_\bullet(\mathrm{F})$ be the associated graded space. The subspace $\mathcal{D}_1 \mathcal{L}_\bullet(\mathrm{F})$ spanned by classical polylogarithms $\Li_n^{\mathcal{L}}(a)$ is denoted by $\mathcal{B}_n(\mathrm{F}).$

Assume that $\Delta=\sum_{1\leq i\leq j}\Delta_{ij}$ for  $\Delta_{ij}\colon \mathcal{L}_{i+j}(\mathrm{F}) \longrightarrow \mathcal{L}_i(\mathrm{F}) \wedge \mathcal{L}_j(\mathrm{F})$. The truncated cobracket is a map $\overline{\Delta}\colon \mathcal{L}(\mathrm{F})\longrightarrow \bigwedge^2\mathcal{L}(\mathrm{F})$ defined by the formula $\overline{\Delta}=\sum_{2\leq i \leq j} \Delta_{ij}.$
In other words, $\overline{\Delta}$ is obtained from $\Delta$ by omitting the component $\mathcal{L}_{1}(\mathrm{F}) \wedge \mathcal{L}_{n-1}(\mathrm{F})$ from the cobracket. Denote by $\textup{coLie}_\bullet(V)$ the cofree graded Lie coalgebra cogenerated by a graded vector space $V$. 

By \cite[Proposition 4.1]{MR22}, the iterated   truncated cobracket 	$\overline{\Delta}^{[d-1]}$ vanishes on  $\mathcal{D}_{d-1}\mathcal{L}_\bullet(\mathrm{F})$ and defines a map 
\begin{align} \label{FormulaIteratedCoproduct}
\overline{\Delta}^{[d-1]}\colon \mathrm{gr}^\mathcal{D}_d\mathcal{L}_{\geq 2}(\mathrm{F}) \longrightarrow \textup{coLie}_d\bigg(\bigoplus_{n\geq 2}\mathcal{B}_n(\mathrm{F})\bigg)\,.
\end{align}

\begin{conjecture}[{Goncharov, \cite[Conjecture 7.6]{Gon01}}] \label{ConjectureDepth}	A linear combination of multiple polylogarithms has depth less than or equal to $d$ if and only if its $d$-th iterated truncated cobracket vanishes. Moreover,  the map $\overline{\Delta}^{[d-1]}$ for $d\geq 1$ is an isomorphism.
\end{conjecture}

We prove the surjectivity part of \autoref{ConjectureDepth}.

\begin{theorem}\label{TheoremTwo}
Assume that the field $\mathrm{F}$ is quadratically closed. Then the map
\[
\overline{\Delta}^{[d-1]}\colon \mathrm{gr}^\mathcal{D}_d\mathcal{L}_{\geq 2}(\mathrm{F}) \longrightarrow \textup{coLie}_d\bigg(\bigoplus_{n\geq 2}\mathcal{B}_n(\mathrm{F})\bigg)
\]
is surjective.
\end{theorem}
\begin{proof}
It is easy to see that 
\[
\overline{\Delta}^{[d-1]}\Bigl(\Li_{n-d\,;\,1,\dots,1}^{\mathcal{L}}(a_1,\dots,a_d)\Bigr)=\sum_{\substack{n_1+n_2+
\dots+n_d=n\\n_i\geq 2}} \Li_{n_1}^{\mathcal{L}}(a_1)\otimes\dots \otimes \Li_{n_d}^{\mathcal{L}}(a_d).
\]
Recall that if $\mathrm{F}$ contains all degree $r$ roots of unity then classical polylogarithms $\Li_n(a)$ satisfy the following {\it distribution relations}:
\[
\Li_n^{\mathcal{L}}(a^r)=r^{n-1}\sum_{\zeta^r=1} \Li_n^{\mathcal{L}}(\zeta a)\,.
\]
It follows that for any $s\in \mathbb{N}$ 
\begin{align*}
&\overline{\Delta}^{[d-1]}\left ( \sum_{x^{2^s}=a_d} \Li_{n-d\,;\,1,\dots,1,1}^{\mathcal{L}}(a_1,\dots,a_{d-1},x) \right)\\
&\quad=\sum_{\substack{n_1+n_2+
\dots+n_d=n\\n_i\geq 2}} 2^{-s(n_d-1)}\Li_{n_1}^{\mathcal{L}}(a_1)\otimes\dots \otimes \Li_{n_d}^{\mathcal{L}}(a_d)
\\
&\quad =\sum_{2\leq n_d\leq n-2d+2} \Bigg (\sum_{\substack{n_1+n_2+
\dots+n_{d-1}=n-n_d\\n_i\geq 2}} \Li_{n_1}^{\mathcal{L}}(a_1)\otimes\dots \otimes \Li_{n_{d-1}}^{\mathcal{L}}(a_{d-1}) \Bigg )\otimes 2^{-s(n_d-1)}\Li_{n_d}^{\mathcal{L}}(a_d)\,.
\end{align*}
From the properties of the Vandermonde determinant it follows that  for every $n_d\geq 2$ the element
\[
 \Bigg (\sum_{\substack{n_1+n_2+
\dots+n_{d-1}=n-n_d\\n_i\geq 2}} \Li_{n_1}^{\mathcal{L}}(a_1)\otimes\dots \otimes \Li_{n_{d-1}}^{\mathcal{L}}(a_{d-1}) \Bigg) \otimes \Li_{n_d}^{\mathcal{L}}(a_d)
\]
lies in the image of $\overline{\Delta}^{[d-1]}$. Continuing in a similar fashion, we conclude that for every $n_1,\dots,n_d\in \mathbb{N}$ the element
\[
\Li_{n_1}^{\mathcal{L}}(a_1)\otimes\dots \otimes \Li_{n_d}^{\mathcal{L}}(a_d)
\]
lies in the image of $\overline{\Delta}^{[d-1]}$. From here the statement follows.
\end{proof}

Assume that Goncharov's depth conjecture holds. It follows from the proof of \autoref{TheoremTwo} that $\mathcal{L}_\bullet(\mathrm{F})$ is generated by functions $\Li_{n-d\,;\,1,\dots,1}(a_1,\dots,a_d).$ Shuffle relations for multiple polylogarithms imply that 
$\Li_{n-d\,;\,1,\dots,1}^{\mathcal{L}}$ can be expressed via $\Li_{0\,;\,n-d+1,\dots,1}^{\mathcal{L}}$ (corresponding to the function $\Li_{n-d+1,\dots,1}$) and functions of lower depth, so \autoref{ConjectureDepth} implies \autoref{ConjectureMain}. 

\Autoref{TheoremTwo} has the following striking corollary.

\begin{corollary} \label{Corollary1}
Let $\mathrm{F}$ be a quadratically closed field. Assume that \autoref{ConjectureDepth} holds for $d=1.$ Then it holds for all $d\geq 1$ and the Lie coalgebra $\mathcal{L}_{\geq 2}(\mathrm{F})$ with cobracket $\overline{\Delta}$ is cofree.
\end{corollary}
\begin{proof}
It is sufficient to prove that (\ref{FormulaIteratedCoproduct}) is an isomorphism: \autoref{ConjectureDepth} and cofreeness of $\mathcal{L}_{\geq 2}(\mathrm{F})$ would follow from the spectral sequence of the filtered complex $\bigwedge^{\bullet}\left(\mathcal{L}_{\geq 2}(\mathrm{F})\right)$;
the filtration on the complex is induced by the depth filtration on $\mathcal{L}_{\geq 2}(\mathrm{F})$. We argue by induction on $d;$ the base case $d=1$ is a tautology. Suppose that for $k\leq d-1$ the map $\overline{\Delta}^{[k-1]}$ is an isomorphism. By  \autoref{TheoremTwo}, it is sufficient to show that $\overline{\Delta}^{[d-1]}$ is injective.  Consider an element  $x\in \mathcal{D}_d\mathcal{L}_{\geq 2}(\mathrm{F})$ such that  $\overline{\Delta}^{[d-1]}(x)=0.$ The map 
\[
\overline\Delta^{[\bullet]}\colon  \mathrm{gr}^\mathcal{D}\mathcal{L}_{\geq 2}(\mathrm{F}) \longrightarrow \textup{coLie}\bigg(\bigoplus_{n\geq 2}\mathcal{B}_n(\mathrm{F})\bigg)
\]
is a morphism of Lie coalgebras, so
\begin{align}\label{FormulaCoproductDelta}
  \sum_{i+j=d}\overline{\Delta}^{[i-1]}\wedge \overline{\Delta}^{[j-1]}(\overline{\Delta}(x))=0.  
\end{align}
By the induction assumption, (\ref{FormulaCoproductDelta}) implies that $\overline{\Delta}(x)$ vanishes in $\bigwedge^2 \mathrm{gr}^\mathcal{D}\mathcal{L}_{\geq 2}(\mathrm{F})=\mathrm{gr}^\mathcal{D}\left(\bigwedge^2 \mathcal{L}_{\geq 2}(\mathrm{F})\right)$ so  
$
\overline{\Delta}(x)\in \mathcal{D}_{d-1}\bigl(\bigwedge^2 \mathcal{L}_{\geq 2}(\mathrm{F})\bigr).
$
The proof follows.
The spectral sequence of the filtered complex implies that Lie coalgebra $\mathcal{D}_{d-1} \mathcal{L}_{\geq 2}(\mathrm{F})$ with cobracket $\overline{\Delta}$ is cofree. Thus there exists $y\in \mathcal{D}_{d-1}\mathcal{L}_{\geq 2}(\mathrm{F})$ such that $\overline{\Delta}(x-y)=0,$ so $x-y\in \mathcal{D}_1\mathcal{L}_{\geq 2}(\mathrm{F})$ by the assumption that \autoref{ConjectureDepth} holds for $d=1$. It follows that $x\in \mathcal{D}_{d-1}\mathcal{L}_{\geq 2}(\mathrm{F}).$
\end{proof}
\medskip
\paragraph{\bf Acknowledgments:} We are very grateful to Alexander Goncharov for suggesting 
\autoref{Corollary1} as a consequence of \autoref{TheoremTwo}.
SC was supported by Deutsche Forschungsgemeinschaft Eigene Stelle grant CH 2561/1-1, for Projektnummer 442093436, HG is grateful to the MPIM Bonn for excellent working conditions.

\bibliographystyle{abbrv}      
\bibliography{surject_bib} 
\end{document}